\documentclass[12pt]{amsart}
\usepackage{enumerate}
\usepackage{amssymb}
\usepackage{bm}
\usepackage{bbm}

\newtheorem{thm}{Theorem}[section]
\newtheorem{lemma}[thm]{Lemma}

\newtheorem{cor}[thm]{Corollary}
\newtheorem{conj}[thm]{Conjecture}

\theoremstyle{definition}
\newtheorem{defn}[thm]{Definition}
\newtheorem{example}[thm]{Example}
\newtheorem{remark}[thm]{Remark}

\newcommand\<{\langle}
\newcommand\0{\mathbf{0}}

\newcommand\CC{\mathbb{C}}
\newcommand\NN{\mathbb{N}}
\newcommand\RR{\mathbb{R}}
\newcommand\ZZ{\mathbb{Z}}
\newcommand\bg{\bm{\gamma}}
\newcommand\bb{\mathbf{b}}

\newcommand\pp{\mathbf{p}}
\newcommand\qq{\mathbf{q}}
\newcommand\cB{\mathcal{B}}

\newcommand\cN{\mathcal{N}}
\newcommand\cP{\mathcal{P}}
\newcommand\cW{\mathcal{W}}
\newcommand\ttt{\mathbf{t}}

\newcommand\minus{\smallsetminus}

\renewcommand\>{\rangle}
\renewcommand\aa{\mathbf{a}}

\newcommand\ceil[1]{\lceil{#1}\rceil}

\newcommand\nim{\textsc{Nim}}
\newcommand\dawson{\textsc{Dawson's Chess}}

\begin{document}%%%%%%%%%%%%%%%%%%%%%%%%%%%%%%%%%%%%%%%%%%%%%%%%%%%%%%
%%%%%%%%%%%%%%%%%%%%%%%%%%%%%%%%%%%%%%%%%%%%%%%%%%%%%%%%%%%%%%%%%%%%%%

\mbox{}\vspace{-1ex}
\title{Algorithms for lattice games}
%\author{Alan Guo}
%\address{Department of Mathematics\\Duke University\\Durham, NC 27708}
%\email{axg@duke.edu}%\vspace{-1.2ex}}
\author{Alan Guo and Ezra Miller}
\address{Department of Mathematics\\Duke University\\Durham, NC 27708}
\curraddr[Guo]{Electrical Engineering \& Computer Science\\MIT\\Cambridge, MA 02139}
\email{aguo@mit.edu, ezra@math.duke.edu}

\begin{abstract}
This paper provides effective methods for the polyhedral formulation
of impartial finite combinatorial games as lattice games
\cite{gmw,latticeGames}.  Given a rational strategy for a lattice
game, a polynomial time algorithm is presented to decide (i)~whether a
given position is a winning position, and to find a move to a winning
position, if not; and (ii)~to decide whether two given positions are
congruent, in the sense of mis\`ere quotient theory
\cite{Pla05,misereQuots}.  The methods are based on the theory of
short rational generating functions \cite{BaWo03}.
\end{abstract}


\date{25 May 2011}

\maketitle

%%%%%%%%%%%%%%%%%%%%%%%%%%%%%%%%%%%%%%%%%%%%%%%%%%%%%%%%%%%%%%%%%%%%%%
\section{Introduction}%%%%%%%%%%%%%%%%%%%%%%%%%%%%%%%%%%%%%%%%%%%%%%%%
%%%%%%%%%%%%%%%%%%%%%%%%%%%%%%%%%%%%%%%%%%%%%%%%%%%%%%%%%%%%%%%%%%%%%%

In \cite{latticeGames}, we reformulated the theory of finite impartial
combinatorial games in the language of combinatorial commutative
algebra and convex rational polyhedral geometry.  In general, the data
provided by a lattice game---a rule set and a game board---do not
allow for easy computation.  Our purpose in this note is to provide
details supporting the claim that \emph{rational strategies}
(Definition~\ref{d:ratstrat}) and \emph{affine stratifications}
(Definition~\ref{d:strat}), two data structures introduced to encode
the set of winning positions of a lattice game \cite{latticeGames},
allow for efficient computation of winning strategies using the theory
of short rational generating functions \cite{BaWo03}.  For example,
given a rational strategy, which is a short rational generating
function for the set of winning positions, computations of simple
Hadamard products decide, in polynomial time, whether any particular
position in a lattice game is a winning or losing position, and which
moves lead to winning positions if the starting position is a losing
position (Theorem~\ref{t:rational}).  Thus the algorithms produce a
winning strategy whose time is polynomial in the input complexity
(Definition~\ref{d:complexity}) once certain parameters, such as
dimension of the lattice game, are held constant.  Other algorithms
extract rational strategies from affine stratifications
(Theorem~\ref{t:strat}) in polynomial time.  None of these results are
hard, given the theory developed by Barvinok and Woods \cite{BaWo03},
but it is worth bringing these methods to the attention of the
combinatorial game theory community.  In addition, the details require
care, and a few results of independent interest arise along the way,
such as Theorem~\ref{t:n-affinestrat}, which says that the complement
of a set with an affine stratification possesses an affine
stratification.

Lattice games that are \emph{squarefree},
\cite[Erratum]{latticeGames}, generalizing the well-known heap-based
\emph{octal games} \cite{GuySmith56} (or see
\cite[Example~6.4]{latticeGames}) such as \nim\ \cite{nim} and
\dawson\ \cite{Dawson34} (with bounded heap size), have tightly
controlled structure in their sets of winning positions under normal
play \cite[Theorem~6.11]{latticeGames}.  However, our efficient
algorithm for computing the set of winning positions in normal-play
squarefree games \cite[Theorem~7.4]{latticeGames} fails to extend to
mis\`ere play, where the final player to move loses.  A~key long-term
goal of this project is to find efficient algorithms for computing
winning strategies in mis\`ere squarefree games, particularly \dawson\
(Remark~\ref{r:dawson}).  As a first step, we conjecture that every
squarefree lattice game---under normal play, ordinary mis\`ere play,
or the more general mis\`ere play allowed by the axioms of lattice
games---possesses an affine stratification
(Conjecture~\ref{c:latticeAffStrat}).  We had earlier conjectured that
every lattice game, squarefree or not, possesses an affine
stratification \cite[Conjecture~8.9]{latticeGames}, but Alex Fink has
disproved that by showing general lattice games to be far from
behaving so calmly~\cite{fink11}.

Other data structures, notably mis\`ere quotients
\cite{Pla05,misereQuots}, that encode winning strategies in families
of games suffer as much as rational strategies do in the face of
Fink's aperiodicity: sufficiently aperiodic sets of P-positions induce
trivial mis\`ere equivalence relations.  Consequently, the bipartite
monoid structure from mis\`ere quotient theory
\cite{Pla05,misereQuots} results in an obvious monoid (a free finitely
generated commutative monoid) with an inscrutable bipartition (the
aperiodic set of P-positions).  Nonetheless, it remains plausible that
mis\`ere quotients extract valuable information about general
squarefree games.  As such, it is useful to continue the quest for
algorithms to compute mis\`ere quotients,
% The latter we accomplish here only for lattice games with extremely
% simple affine stratifications (Theorem~\ref{t:1as}).
which have already led to advances in our understanding of octal games
with finite quotients \cite{Pla05,misereQuots}; see \cite{Wei09} for
recent progress in some cases.  Theorem~\ref{t:disting} is a step
toward computing infinite mis\`ere quotients from rational strategies:
it gives an efficient test for mis\`ere equivalence.

The results in this note reduce the problem of efficiently finding a
winning strategy for a finite impartial combinatorial game to the
problem of efficiently computing an affine stratification, or even
merely a rational strategy.  Neither of these would complete the
solution of, say, \dawson\ in polynomial time, because the
polynomiality here assumes bounded heap size, but they would be
insightful first~steps.

%%%%%%%%%%%%%%%%%%%%%%%%%%%%%%%%%%%%%%%%%%%%%%%%%%%%%%%%%%%%%%%%%%%%%%
\subsection*{Acknowledgements}%%%%%%%%%%%%%%%%%%%%%%%%%%%%%%%%%%%%%%%%

The seminal idea for this work emerged from extended discussions at
Dalhousie University with the participants of Games-At-Dal-6 (July,
2009).  Thanks also go to Sasha Barvinok for a discussion of
computations with generating functions, and for pointing out his
beautiful survey article \cite{Bar06}.  Partial funding was provided
by NSF Career DMS-0449102 = DMS-1014112 and NSF DMS-1001437.

%%%%%%%%%%%%%%%%%%%%%%%%%%%%%%%%%%%%%%%%%%%%%%%%%%%%%%%%%%%%%%%%%%%%%%
\section{Lattice games}\label{s:boards}%%%%%%%%%%%%%%%%%%%%%%%%%%%%%%%
%%%%%%%%%%%%%%%%%%%%%%%%%%%%%%%%%%%%%%%%%%%%%%%%%%%%%%%%%%%%%%%%%%%%%%

Precise notions of complexity require a review of the axioms for
lattice games from \cite{latticeGames}.  To that end, fix a pointed
rational convex polyhedron $\Pi \subset \RR^d$ with recession cone $C$
of dimension~$d$.  Write $\Lambda = \Pi \cap \ZZ^d$ for the set of
integer points in~$\Pi$.

\begin{defn}\label{d:tangentcone}
Given an extremal ray $\rho$ of a cone $C$, the \emph{negative tangent
cone} of $C$ along $\rho$ is $-T_\rho C = -\bigcap_{H \supset \rho}
H_+$ where $H_+ \supseteq C$ is the positive closed halfspace bounded
by a supporting hyperplane $H$ for $C$.
\end{defn}

\begin{defn}\label{d:ruleset}
A finite subset $\Gamma \subset \ZZ^d \minus \{\0\}$ is a \emph{rule
set} if
\begin{enumerate}
\item%
there exists a linear function on $\RR^d$ that is positive on
$\Gamma \cup C \minus \{\0\}$; and
\item%
for each ray $\rho$ of $C$, some vector $\bg_\rho \in \Gamma$ lies
in the negative tangent cone $-T_\rho C$.
\end{enumerate}
\end{defn}

\begin{defn}\label{d:latticegame}
Given the polyhedral set $\Lambda = \Pi \cap \ZZ^d$, fix a rule set
$\Gamma$.
\begin{itemize}
\item%
A \emph{game board} $\cB$ is the complement in $\Lambda$ of a finite
$\Gamma$-order ideal in $\Lambda$ called the set of \emph{defeated
positions}.
\item%
A \emph{lattice game} is defined by a game board and a rule set.
\end{itemize}
\end{defn}

\begin{defn}[\textbf{Input complexity of a lattice game}]\label{d:complexity}
Let $(\Gamma,\cB)$ be a lattice game with rule set $\Gamma$ and game
board $\cB$. $\Gamma$ may be represented as a $d \times n$ matrix with
entries $\gamma_{ij}$ for $1 \le i \le d$ and $1 \le j \le n$, where
$n = |\Gamma|$.  The game board $\cB$ may be represented by the $m$
generators of the finite $\Gamma$-order ideal, hence a $d \times m$
matrix with entries $a_{ij}$ for $1 \le i \le d$ and $1 \le j \le
m$.  The \emph{input complexity} of the lattice game is the number of
bits needed to represent these $d(m + n)$ numbers, namely
$$%
  d(m + n) +
  \sum_{i=1}^d \Big(
  \sum_{j=1}^n \log_2| \gamma_{ij}| +
  \sum_{j=1}^m \log_2 |a_{ij}| \Big).
$$
\end{defn}

%%%%%%%%%%%%%%%%%%%%%%%%%%%%%%%%%%%%%%%%%%%%%%%%%%%%%%%%%%%%%%%%%%%%%%
\section{Rational strategies as data structures}\label{s:rational}%%%%
%%%%%%%%%%%%%%%%%%%%%%%%%%%%%%%%%%%%%%%%%%%%%%%%%%%%%%%%%%%%%%%%%%%%%%

\begin{defn}\label{d:ratgenfun}
For $A \subseteq \ZZ^d$, the \emph{generating function} for $A$ is the
formal series
$$%
  f(A;\ttt) = \sum_{\aa \in A} \ttt^\aa.
$$
\end{defn}

\begin{defn}\label{d:ratstrat}
A \emph{rational strategy} for a lattice game is a generating
function for the set of P-positions of the form
$$%
  f(A;\ttt) = \sum_{i \in I} \alpha_i
  \frac{ \ttt^{\pp_i}}{(1-\ttt^{\aa_{i1}})\cdots(1-\ttt^{\aa_{ik(i)}})},
$$
for some finite set $I$, nonnegative integers $k(i)$, rational numbers
$\alpha_i$, along with vectors $\pp_i,\aa_{ij} \in \ZZ^d$ and
$\aa_{ij} \ne \0$ for all $i,j$ \cite[Definition~8.1]{latticeGames}.
A rational strategy is \emph{short} if the number $|I|$ of indices is
bounded by a polynomial in the input complexity.
\end{defn}

\begin{defn}[\textbf{Complexity of short rational generating functions}]
Fix a positive integer~$k$.  Let $A \subseteq \ZZ^d$ and
$$%
 f(A;\ttt)
 =
 \sum_{i\in I}\alpha_i\frac{\ttt^{\pp_i}}{(1-\ttt^{\aa_{i1}})\cdots(1-\ttt^{\aa_{ik}})}
$$
for some vectors $\pp_i,\aa_{ij} \in \ZZ^d$ and $\aa_{ij} \ne 0$ for
all~$i,j$.  If $\pp_i = (p_{i1},\ldots,p_{id})$ and $\aa_{ij} =
(a_{ij1},\ldots,a_{ijd})$ for all~$i,j$, and $\alpha_i$ is given as a
ratio $\pm \frac{\alpha'_i}{\alpha''_i}$ of positive
% relatively prime?
integers, then the \emph{complexity} of $f(A;\ttt)$ is the number
$$%
  |I|(1 + d + kd) + \sum_{i \in I}\Big(\log_2 \alpha'_i + \log_2 \alpha''_i
  + \sum_{j=1}^d \log_2 |p_{ij}| +
  \sum_{j=1}^k \sum_{r=1}^d \log_2 |a_{ijr}| \Big).
$$
\end{defn}

\begin{defn}[{\cite[Definition~3.2]{BaWo03}}]
For Laurent power series
$$%
  f_1(\ttt) = \sum_{\aa \in \ZZ^d} \beta_{1\aa}\ttt^\aa
  \text{ \ and \ }
  f_2(\ttt) = \sum_{\aa \in \ZZ^d} \beta_{2\aa}\ttt^\aa
$$
in $\ttt \in \CC^d$, the \emph{Hadamard product} $f = f_1 \star f_2$
is the power series
$$%
  f(\ttt) = \sum_{\aa\in\ZZ^d}(\beta_{1\aa}\beta_{2\aa})\,\ttt^\aa.
$$
\end{defn}

\begin{lemma}\label{l:hadamard}
Fix $k$.  Let $A,B \subseteq \ZZ^d$ lie in a common pointed rational
cone~$C$.  If $f(A;\ttt)$ and $f(B;\ttt)$ are rational generating
functions with $\leq k$ denominator binomials in each, then there is
an algorithm for computing $f(A;\ttt) \star f(B;\ttt)$ as a rational
generating function in polynomial time in the complexity of the
generating functions.
\end{lemma}
\begin{proof}
Choose an affine linear function $\ell$ that is negative on $C$.
Write
\begin{eqnarray*}
f(A;\ttt)
&=& \sum_{i \in I} \alpha_i
    \frac{\ttt^{\pp_i}}{(1-\ttt^{\aa_{i1}})\cdots(1-\ttt^{\aa_{ik}})}
\\
f(B;\ttt)
&=& \sum_{j \in J} \beta_i
    \frac{\ttt^{\qq_i}}{(1-\ttt^{\bb_{j1}})\cdots(1-\ttt^{\bb_{jk}})},
\end{eqnarray*}
where $\pp_i,\qq_i \in \ZZ^d$, $\aa_{ir},\bb_{jr} \in C$ for
all~$i,j,r$.  Since $\<\ell,\aa_{ir}\> < 0$ and $\<\ell,\bb_{jr}\> <
0$ for all $i,j,r$, by Lemma 3.4 of \cite{BaWo03} we can compute
$$%
  \frac{ \ttt^{\pp_i} }{(1-\ttt^{\aa_{i1}})\cdots(1-\ttt^{\aa_{ik}})}
  \star
  \frac{ \ttt^{\qq_i} }{(1-\ttt^{\bb_{j1}})\cdots(1-\ttt^{\bb_{jk}})}
$$
in polynomial time for each~$i,j$.  Since the Hadamard product is
bilinear, it follows that we can compute $f(A;\ttt) \star f(B;\ttt)$
in polynomial time as well.
\end{proof}

\begin{thm}\label{t:rational}
Any rational strategy for a lattice game produces algorithms for
\begin{itemize}
\item%
determining whether a position is a P-position or an N-position, and
\item%
computing a legal move to a P-position, given any N-position.
\end{itemize}
These algorithms run in polynomial time if the rational strategy is
short.
\end{thm}
\begin{proof}

Suppose we wish to determine whether $\pp \in \cB$ is a P-position or
an N-position.  Let $f(\cP;\ttt)$ be a rational strategy for the
lattice game.  By definition, $\cP$ and $\pp$ both lie in the
cone~$C$.  It follows from Lemma \ref{l:hadamard} that we can compute
$f(\cP \cap \pp;\ttt) = f(\cP;\ttt) \star \ttt^\pp$ in $O(\iota^c)$
time, where $\iota$ is the complexity of $f(\cP;\ttt)$ and $\ttt^\pp$,
and $c$ is some positive integer.  We get
$$%
  f(\cP \cap \pp;\ttt)
  =
  \left\{
    \begin{array}{ll}
      \ttt^\pp & \text{if } \pp \in \cP
      \\
          0    & \text{if } \pp \in \cN.
    \end{array}
  \right.
$$
Given an N-position~$\qq$, simply apply this algorithm to all
positions $\qq - \bg$ for each legal move $\bg \in \Gamma$.  Since
$\qq \in \cN$, at least one $\qq - \bg$ lies in $\cP$, hence this
procedure will end in $O(\iota^c|\Gamma|)$ time.
\end{proof}

\begin{remark}\label{r:dawson}
The eventual goal of this project is to solve \dawson.  That is, given
any position in \dawson, we desire efficient algorithms to determine
whether the next player to move has a winning strategy, and if so, to
find one.  This is equivalent to determining whether a given position
$\pp$ is a P-position or an N-position.  If $\pp \in \cN$, then the
next player to move indeed has a winning strategy by moving the game
to a P-position.  This is the problem of determining those $\bg \in
\Gamma$ for which $\pp - \bg$ lies in $\cP$.  By
Theorem~\ref{t:rational}, we can do all of this if we have a \dawson\
rational strategy for heaps of sufficient size.  Alas, it is not known
whether rational strategies exist for general squarefree games, or
even for \dawson.
\end{remark}

\begin{conj}
Every squarefree lattice game possesses a rational strategy.
\end{conj}

It is known that arbitrary lattice games need not possess rational
strategies \cite{fink11}.  The smallest known counterexample is
on~$\NN^3$; its rule set has size~28.

\begin{remark}\label{r:dawson'}
The question, then, is how to find a rational strategy for \dawson.
Observe that a fixed lattice game structure only suffices to encode a
heap game for heaps of bounded size.  Let $G_n$ denote the lattice
game corresponding to \dawson\ with heaps of size at most~$n$.  If we
can find the rational strategy for any given~$n$, then this is good
enough, although we must be careful about the complexity of finding
such rational strategies as a function of~$n$.  In the next sections,
we shall see that affine stratifications serve as data structures from
which to extract the rational strategy in polynomial time.  Thus the
problem will be reduced to finding affine stratifications for $G_n$
for all~$n$, and there is hope that some regularity might arise, as
$n$ grows, to allow the possibility of computing them in time
polynomial in~$n$.
\end{remark}

%%%%%%%%%%%%%%%%%%%%%%%%%%%%%%%%%%%%%%%%%%%%%%%%%%%%%%%%%%%%%%%%%%%%%%
\section{Affine stratifications as data structures}\label{s:affine}%%%
%%%%%%%%%%%%%%%%%%%%%%%%%%%%%%%%%%%%%%%%%%%%%%%%%%%%%%%%%%%%%%%%%%%%%%

\begin{defn}[{\cite[Definition~8.6]{latticeGames}}]\label{d:strat}
An \emph{affine stratification} of a subset $\mathcal{W} \subseteq
\ZZ^d$ is a partition
$$%
  \mathcal{W} = \biguplus_{i=1}^r W_i
$$
of $\mathcal{W}$ into a disjoint union of sets $W_i$, each of which is
a finitely generated module for an affine semigroup $A_i \subset
\ZZ^d$; that is, $W_i = F_i + A_i$, where $F_i \subset \ZZ^d$ is a
finite set.  An \emph{affine stratification of a lattice game} is an
affine stratification of its set of P-positions.
\end{defn}

The choice to require an affine stratification of~$\cP$, as opposed
to~$\cN$, may seem arbitrary, but in the end these are equivalent, due
to the following result.

\begin{thm}\label{t:n-affinestrat}
If $A$ and $B \subset A$ both possess affine stratifications, then $A
\minus B$ possesses an affine stratification.
\end{thm}

The plan for Theorem~\ref{t:n-affinestrat} is to show that removing a
translated normal affine semigroup (an affine semigroup is
\emph{normal} if it is the intersection of a real cone with a lattice;
see \cite[Chapter~7]{cca}) from a normal affine semigroup yields an
affinely stratified set, and intersecting affinely stratified sets
results in an affinely stratified~set.

\begin{lemma}\label{l:carving}
Suppose $B$ is the intersection of a rational convex polyhedron and a
subgroup of $\ZZ^d$.  If $A$ is a normal affine semigroup and $\bb + A
\subset B$ for some $\bb \in B$, then $B \minus (\bb + A)$ has an
affine stratification.
\end{lemma}
\begin{proof}
First we assume that $\bb = \0$ and that $B$ is a normal affine
semigroup and $\RR_{\geq 0} A = \RR_{\geq 0} B$.  Since $A \subset B$,
that means $\ZZ A$ is a sublattice of $\ZZ B$ in $\ZZ^d$, hence $B$
can be written as a finite disjoint union of cosets of $A$.

Now, suppose $B$ is an arbitrary intersection of a rational convex
polyhedron $\Pi_B$ and a lattice $L$, and $\bb \in B$ is arbitrary.
We will reduce to the previous case by ``carving'' away pieces of $B$
that do not lie in $\RR_{\geq 0}A$.  Suppose $\RR_{\geq 0}A$ has a
facet (a $(d-1)$-dimensional face) which is not contained in a facet
of $\Pi_B$.  Let $H$ be the bounding hyperplane of this facet and
$H_-$ the corresponding negative halfspace (the half that is outside
of $\RR_{\geq 0}A$).  Then $H_- \cap \Pi_B$ is a rational convex
polyhedron.  to reduce the number of facets of $\RR_{\geq 0}A$ which
do not lie in a facet of $\Pi_B$.  Thus we have ``carved out'' a piece
$H_- \cap \Pi_B$ of $\Pi_B$.  By \cite[Lemma~2.4] {affineStrat}, $H_-
\cap \Pi_B \cap L$ is a finitely generated module over an affine
semigroup.  Now replace $\Pi_B$ with $\Pi_B \minus H_-$ and repeat.
Each time we repeat the argument, we carve out a piece of the original
$\Pi_B$ which has an affine stratification, and furthermore we reduce
the number of facets of $\RR_{\geq 0}A$ that do not lie in the current
$\Pi_B$.  Eventually we reduce to the case where each facet of
$\RR_{\geq 0}A$ lie in some facet of $\Pi_B$, which is actually the
first case above where $\Pi_B$ is a cone and $\bb = \0$.  By
\cite[Corollary~2.8]{affineStrat}, the union of these pieces possesses
an affine stratification.

There is a degenerate case when $A$ is not $d$-dimensional, but then
we may reduce to a lower dimension by carving away $\ZZ^d \minus A$.
\end{proof}

\begin{lemma}\label{l:intersectaffinestrat}
If $\cW$ and $\cW'$ have affine stratifications, then $\cW \cap \cW'$
has an affine stratification.
\end{lemma}
\begin{proof}
By \cite[Theorem~2.6]{affineStrat}, we may write
$$%
  \cW = \biguplus_{i=1}^r W_i \text{\ \ \ and \ \ \ } \cW' =
  \biguplus_{j=1}^s W'_i
$$
where each $W_i$ and $W'_j$ is a translate of a normal affine
semigroup.  Therefore, it suffices to show that the intersection of a
translate of a normal affine semigroup with a translate of another
normal affine semigroup has an affine stratification, for the union of
all of these intersections would then have an affine stratification,
by \cite[Corollary~2.8]{affineStrat}.

Suppose our two translates are $\aa_1 + A_1$ and $\aa_2 + A_2$.  If
their intersection is empty, then trivially it has an affine
stratification, so we may assume that there is some $\aa \in (\aa_1 +
A_1) \cap (\aa_2 + A_2)$.  Then $\aa_1 - \aa + \ZZ A_1 = \ZZ A_1$ and
$\aa_2 - \aa + \ZZ A_2 = \ZZ A_2$.  Therefore
\begin{eqnarray*}
(\aa_1 + \ZZ A_1) \cap (\aa_2 + \ZZ A_2) &=&
\aa + (\aa_1 - \aa + \ZZ A_1) \cap (\aa_2 - \aa + \ZZ A_2) \\
&=& \aa + (\ZZ A_1 \cap \ZZ A_2),
\end{eqnarray*}
i.e., the intersection of the cosets is itself a coset of a
lattice.  Moreover, the intersection $(\aa_1 + \RR_{\geq 0} A_1) \cap
(\aa_2 + \RR_{\geq 0} A_2)$ is a
polyhedron.  By~\cite[Lemma~2.4]{affineStrat}, since $A_1$ and $A_2$
are normal, we have
\begin{align*}
(\aa_1 + A_1) \cap (\aa_2 + A_2)
&= ((\aa_1 + \RR_{\geq 0} A_1) \cap (\aa_1 + \ZZ A_1)) \cap ((\aa_2 +
   \RR_{\geq 0} A_2) \cap (\aa_2 + \ZZ A_2))
\\
&= ((\aa_1 + \RR_{\geq 0} A_1) \cap (\aa_2 + \RR_{\geq 0} A_2)) \cap
   ((\aa_1 + \ZZ A_1) \cap (\aa_2 + \ZZ A_2))
\end{align*}
is an intersection of a polyhedron with a coset of a lattice and hence
is a finitely generated module over an affine semigroup.  In
particular, the intersection has an affine stratification.
\end{proof}

%begin{proof}[Proof of Theorem~\ref{t:n-affinestrat}]
\begin{trivlist}\item\textbf{Proof of Theorem~\ref{t:n-affinestrat}.}
First, assume $A$ is a normal affine semigroup.  Suppose
$$%
  B = \biguplus_{i=1}^r B_i
$$
where each $B_i$ is a translate of a normal affine semigroup.  By
Lemma~\ref{l:carving}, each $A \minus B_i$ has an affine
stratification.  Therefore, by Lemma~ \ref{l:intersectaffinestrat}, $A
\minus B = A \minus (\biguplus_{i=1}^r B_i) = \bigcap_{i=1}^r (A
\minus B_i)$ has an affine stratification.  For the general case where
$A$ has an affine stratification, each $A_i$ reduces to the previous
case, and then we obtain the result by taking the union.\hfill$\Box$
\end{trivlist}\smallskip
%end{proof}

\begin{conj}\label{c:latticeAffStrat}
Every squarefree lattice game possesses an affine stratification.
\end{conj}

\begin{example}
Consider again the game of \nim\ with heaps of size at most~$2$.  An
affine stratification for this game is $\cP = 2\NN^2$; that is, $\cP$
consists of all nonnegative integer points with both coordinates even.
\end{example}

\begin{example}
The mis\`ere lattice game on~$\NN^5$ whose rule set forms the
columns~of
$$%
\newcommand\ph{\phantom{-}}
\Gamma = \left[
         \begin{array}{@{}*{8}{r@{\ \,}}@{}r}
           \!\ph1&\ph0&\ph0&\ph0&-1& 0& 0& 0\\
           \!   0&   1&   0&   0& 1&-1& 0& 0\\
           \!   0&   0&   0&   0& 0& 1&-1& 0\\
           \!   0&   0&   1&   0& 0& 0& 1&-1\\
           \!   0&   0&   0&   1& 0& 0& 0& 1
         \end{array}
         \right]
$$
was one of the motivations for the definitions in \cite{latticeGames}
because the illustration of the winning positions in this lattice game
provided by Plambeck and Siegel \cite[Figure~12]{misereQuots} possesses
an interesting description as an affine stratification.  An explicit
description can be found in \cite[Example~8.13]{latticeGames}.
\end{example}

In what follows, we define the complexity of an affine stratification
to be the complexity of the generators and the affine semigroups
involved.  Roughly speaking, the complexity of an integer~$k$ is its
binary length (more precisely, $1 + \ceil{\log_2 k}$), so the
complexity is roughly the sum of the binary lengths of the integer
entries of the generators in the finite sets~$F_i$ and the
coefficients of the vectors generating the affine semigroups~$A_i$;
see \cite[Section~2]{Bar06} for additional details.  To say that an
algorithm is \emph{polynomial time when the dimension~$d$ is fixed}
means that the running time is bounded by $\iota^{\phi(d)}$ for some
fixed function~$\phi$, where $\iota$ is the complexity.

\begin{defn}[\textbf{Complexity of an affine semigroup}]
Fix an affine semigroup $A = \NN\{\aa_1,\ldots,\aa_n\}$ in $\ZZ^d$.
Let $\aa_i = (a_{1i},\ldots,a_{di})$.  Then $A$ may be represented by a
$d \times n$ matrix with entries $a_{ij}$.  The \emph{complexity} of
$A$ is the number of bits needed to represent these $dn$ numbers,
which is equal to
$$%
  dn + \sum_{i=1}^n\sum_{j=1}^d \log_2 |a_{ij}|.
$$
\end{defn}

\begin{defn}[\textbf{Complexity of an affine stratification}]
Let
$$%
  \cP = \biguplus_{i=1}^r W_i
$$
be an affine stratification, where $W_i = F_i + A_i$ for some affine
semigroup $A_i \subset \ZZ^d$ and finite set $F_i \subset \ZZ^d$.  Let
$m_i = |F_i|$ and $F_i = \{\bb_{i1},\ldots,\bb_{im_i}\}$ where
$\bb_{ij} = (b_{ij1},\ldots,b_{ijd})$, and let $A_i =
\NN\{\aa_{i1},\ldots,\aa_{in}\}$, where $\aa_{ij} = (a_{ij1},
\ldots,a_{ijd})$.  The \emph{complexity} of the affine stratification
is the number
$$%
  d\Big(nr + \sum_{i=1}^r m_i\Big)
  + \sum_{i=1}^r \sum_{s=1}^d
    \Big(\sum_{j=1}^n\log_2|a_{ijs}| + \sum_{j=1}^{m_i}\log_2|b_{ijs}|\Big)
$$
of bits needed to represent each $F_i$ and $A_i$.
\end{defn}

\begin{remark}
The existence of affine stratifications as in
\cite[Conjecture~8.9]{latticeGames} is equivalent to the same
statement with the extra hypothesis that the rule set generates a
saturated (also known as ``normal'') affine semigroup.  There are also
a number of ways to characterize the existence of affine
stratifications, in general \cite[Theorem~2.6]{affineStrat}, using
various combinations of hypotheses such as normality of the affine
semigroups involved, or disjointness of the relevant unions.  However,
some of these freedoms increase complexity in untamed ways, and are
therefore unsuitable for efficient algorithmic purposes.
Definition~\ref{d:strat} characterizes the notion of affine
stratification in the most efficient terms, where algorithmic
computation of rational strategies is concerned; allowing the unions
to overlap would make it easier to find affine stratifications, but
harder to compute rational strategies from them.
\end{remark}

%%%%%%%%%%%%%%%%%%%%%%%%%%%%%%%%%%%%%%%%%%%%%%%%%%%%%%%%%%%%%%%%%%%%%%
\section{Computing rational strategies from affine stratifications}\label{s:compute}
%%%%%%%%%%%%%%%%%%%%%%%%%%%%%%%%%%%%%%%%%%%%%%%%%%%%%%%%%%%%%%%%%%%%%%

In this section, we prove the following.

\begin{thm}\label{t:strat}
A rational strategy can be algorithmically computed from any affine
stratification, in time polynomial in the input complexity of the
affine stratification when the dimension~$d$ is fixed and the numbers
of module generators over the semigroups $A_i$ are uniformly bounded
above.
\end{thm}

%The main points here are that
%\begin{itemize}
%\item%
%short rational generating functions for the affine semigroups $A_i$
%can be computed in polynomial time when the dimension is fixed; and
%\item%
%short rational generating functions for a union of affine
%subsemigroups of~$\NN^d$ can be computed from their short rational
%generating functions in polynomial time when the dimension is fixed.
%\end{itemize}
%Note, however, that the complexity of the denominator is exponential
%in the total number of generators of the $A_i$-modules.  Something to
%check: does the algorithm simplify when the denominators are the same?
%Even if not, the \emph{answer} definitely does, since a finitely
%generated $A_i$-module must have a rational generating function with
%the same denominator as $f(A_i;\ttt)$ itself.  Thus the complexity of
%the denominator ought to be exponential in the number of affine
%semigroups~$A_i$, not the total number of generators.

The proof of the theorem requires a few intermediate results, the
point being simply to keep careful track of the complexities of the
constituent elements of affine stratifications.

\begin{lemma}\label{l:union}
Fix $k,d \in \NN$.  Let $A, B \subseteq \ZZ^d$ lie in the same pointed
rational cone~$C$.  If $f(A;\ttt)$ and $f(B;\ttt)$ are short rational
generating functions with $\le k$ binomials in their denominators,
then for some $c \in \NN$ there is an $O(\iota^c)$ time algorithm for
computing the rational function $f(A \cup B;\ttt)$, where $\iota$ is
an upper bound on the complexity of $f(A;\ttt)$ and $f(B;\ttt)$.  If
$A$ and $B$ are disjoint, then the complexity of $f(A \cup B;\ttt)$ is
bounded by~$2\iota$, and $f(A \cup B;\ttt)$ can be computed in
$O(\iota)$ time.
\end{lemma}
\begin{proof}
This follows from the fact that
$$%
  f(A \cup B;\ttt) = f(A;\ttt) + f(B;\ttt) - f(A \cap B;\ttt)
$$
and that $f(A \cap B;\ttt) = f(A;\ttt) \star f(B;\ttt)$ can be
computed in polynomial time, by Lemma~\ref{l:hadamard}.
\end{proof}

\begin{cor}\label{c:union}
Fix $k,d \in \NN$.  Let $A_1,\ldots,A_m \subseteq \ZZ^d$ lie in the
same pointed rational cone~$C$.  If $f(A_1;\ttt),\ldots,f(A_r;\ttt)$
are rational generating functions with $\le k$ binomials in their
denominators, and $A = A_1 \cup \cdots \cup A_m$, then for some $c \in
\NN$ there is an $O(2^m\iota^c)$ time algorithm for computing
$f(A;\ttt)$ as a rational generating function, where $\iota$ is an
upper bound on the complexity of $f(A_1;\ttt),\ldots,f(A_r;\ttt)$.  If
the $A_i$ are pairwise disjoint, then the complexity bound is
$O(m\iota)$.
\end{cor}
\begin{proof}
This follows from Lemma~\ref{l:union} and the fact that the number of
binomials in the denominators in the rational generating functions may
increase by a factor of up to~2 after computing each union.  If the
$A_i$ are pairwise disjoint, then
$$%
  f(A;\ttt) = \sum_{i=1}^m f(A_i;\ttt)
$$
and no intersections need to be computed.
\end{proof}

\begin{lemma}\label{l:affgenfun}
Fix $n$ and $d$.  If $A \subseteq \ZZ^d$ is a pointed affine semigroup
generated by $n$ integer vectors and has complexity $\iota$, then for
some positive integer $c$ there is an $O(\iota^c)$ time algorithm for
computing $f(A;\ttt)$.
\end{lemma}
\begin{proof}
Let $A = \NN\{\aa_1,\ldots,\aa_n\}$.  It is algorithmically easy to
embed $A$ into $\NN^d$: if $A$ has dimension $d$, then find $d$
linearly independent facets and take their integer inner normal
vectors as the columns of the embedding $\nu$; if $A$ has dimension
$d' < d$, then find $d'$ linearly independent facets and any $d - d'$
linear integer functions that vanish on~$A$.  Use the discussion of
\cite[Section~7.3]{BaWo03} to compute $f(\nu(A);\ttt)$.  Then apply
$\nu^{-1}$ to the exponents in $f(\nu(A);\ttt)$ to get $f(A;\ttt)$.
\end{proof}

\begin{lemma}\label{l:module}
Fix $d$. Let $W = F + A$, where $A \subseteq \ZZ^d$ is a pointed affine
semigroup with complexity $\iota$ and $F \subseteq \ZZ^d$ is a finite
set with $|F| = m$.  For some $c \in \NN$ there is an $O(2^m\iota^c)$
time algorithm for computing $f(W;\ttt)$ as a rational function.
\end{lemma}
\begin{proof}
Let $F = \{\bb_1,\ldots,\bb_m\}$.  Since $F$ is finite, any linear
function that is positive on $A \minus \{0\}$ is bounded below on
$W$.  Therefore, there exists a pointed rational cone that contains
each $\bb_j + A$.  For each $j$, $f(\bb_j + A; \ttt) = \ttt^{\bb_j}
f(A;\ttt)$, each of which has complexity $O(\iota)$ and can be
computed in $O(\iota^{c'})$ time, for some $c' > 0$, by
Lemma~\ref{l:affgenfun}.  Since $W$ is the union of the $\bb_j + A$,
it follows from Corollary~\ref{c:union} that $f(W;\ttt)$ can be
computed in $O(2^m\iota^c)$.
\end{proof}

We now return to proving our main theorem.

\smallskip
%begin{proof}[Proof of Theorem~\ref{t:strat}]
\begin{trivlist}\item\textbf{Proof of Theorem~\ref{t:strat}.}
Write
$$%
  \cP = \biguplus_{i=1}^r W_i
$$
where $W_i = F_i + A_i$ for affine semigroups $A_i \subseteq \ZZ^d$
and finite sets $F_i \subseteq \ZZ^d$.  Let $\iota$ be an upper bound
on the complexity of each of the $A_i$.  Since the sizes of the $F_i$
are fixed, by Lemma~\ref{l:module} we can compute each $f(W_i;\ttt)$
in $O(\iota^c)$ time, for some positive integer~$c$.  Since the $W_i$
are pairwise disjoint, by Corollary~\ref{c:union} we can compute
$f(\cP;\ttt)$ in $O(r\iota^c)$ time.\hfill$\Box$
\end{trivlist}\medskip
%end{proof}

There is little hope that the complexity of calculating affine
stratifications---or even merely rational strategies---should be
polynomial in the input complexity when certain parameters are not
fixed.  Indeed, the complexity of the generating function for an
affine semigroup fails to be polynomial in the number of its
generators.  Thus it makes sense to restrict complexity estimates to
lattice games with rule sets of fixed complexity.  On the other hand,
there is hope that the complexity of an affine stratification should
be bounded by the complexity of the rule set.  Therefore, once the
complexity of the rule set has been fixed, the algorithms dealing with
affine stratifications could be polynomial.

%%%%%%%%%%%%%%%%%%%%%%%%%%%%%%%%%%%%%%%%%%%%%%%%%%%%%%%%%%%%%%%%%%%%%%
\section{Determining mis\`ere congruence}\label{s:cong}%%%%%%%%%%%%%%%
%%%%%%%%%%%%%%%%%%%%%%%%%%%%%%%%%%%%%%%%%%%%%%%%%%%%%%%%%%%%%%%%%%%%%%

This section provides an algorithm to determine whether two given
positions are mis\`ere congruent.  The notion of mis\`ere congruence
is simply the translation of ``indistinguishability''
\cite{Pla05,misereQuots} into the language of lattice games.

\begin{defn}
Two positions $\pp$ and $\qq$ are \emph{(mis\`ere) congruent} if
$$%
  (\pp + C) \cap \cP - \pp = (\qq + C) \cap \cP - \qq.
$$
\end{defn}

\begin{thm}\label{t:disting}
Given a rational strategy $f(\cP;\ttt)$ and $\pp,\qq \in \cB$, there
is a polynomial time algorithm for determining whether $\pp$ and~$\qq$
are mis\`ere congruent.
\end{thm}
\begin{proof}
Let $S_\pp = (\pp + C) \cap \cP - \pp$ and $S_\qq = (\qq + C) \cap \cP
- \qq$.  Since $\pp \in C$, we have $\pp + C \subseteq C$, and $\cP
\subseteq C$ by definition, so we may apply Lemma \ref{l:hadamard} to
compute $f((\pp + C) \cap \cP;\ttt)$ in polynomial time.  Then we can
compute $f(S_\pp;\ttt)$ in polynomial time since $f(S_\pp;\ttt) =
\ttt^{-\pp}f((\pp + C) \cap \cP;\ttt)$.  Similarly, we compute
$f(S_\qq;\ttt)$ in polynomial time.  Then $\pp$ and $\qq$ are
congruent if and only if $f(S_\pp;\ttt) - f(S_\qq;\ttt) = 0$.
\end{proof}

\begin{cor}
Given an affine stratification of a lattice game, there is a
polynomial time algorithm that decides on the mis\`ere congruence of
any pair of positions.
\end{cor}
\begin{proof}
In polynomial time, Theorem \ref{t:strat} produces a rational strategy
for the lattice game and then Theorem~\ref{t:disting} decides on the
congruence.
\end{proof}

%%%%%%%%%%%%%%%%%%%%%%%%%%%%%%%%%%%%%%%%%%%%%%%%%%%%%%%%%%%%%%%%%%%%%%
\enlargethispage{1ex}%%%%%%%%%%%%%%%%%%%%%%%%%%%%
%%%%%%%%%%%%%%%%%%%%%%%%%%%%%%%%%%%%%%%%%%%%%%%%%%%%%%%%%%%%%%%%%%%%%%

%%%%%%%%%%%%%%%%%%%%%%%%%%%%%%%%%%%%%%%%%%%%%%%%%%%%%%%%%%%%%%%%%%%%%%
\end{document}